\documentclass[reqno]{amsart}
\usepackage{mathrsfs}
\usepackage{amssymb}
\usepackage{latexsym}
\usepackage{yfonts}
\usepackage{natbib}
\usepackage[english]{babel}
\usepackage{graphicx}

\bibpunct{[}{]}{;}{n}{,}{,}

\newtheorem{te}{Theorem}[section]

\theoremstyle{lem}

\theoremstyle{coro}
\newtheorem{coro}[te]{Corollary}

\theoremstyle{definition}

\theoremstyle{os}
\newtheorem{os}[te]{Remark}

\numberwithin{equation}{section}

\allowdisplaybreaks

\begin{document}

\title[Hyper-Cauchy laws]{Higher-order Laplace equations\\ and hyper-Cauchy distributions}

\author{Enzo Orsingher} 
\address{Dipartimento di Scienze Statistiche, Sapienza University of Rome - p.le A. Moro 5, 00185 Roma, Italy}
\email{enzo.orsingher@uniroma1.it}

\author{Mirko D'Ovidio}
\address{Dipartimento di Scienze di Base e Applicate per l'Ingegneria, Sapienza University of Rome - via A. Scarpa 10, 00161 Roma, Italy}
\email{mirko.dovidio@uniroma1.it}
 

\date{\today}


\begin{abstract}
In this paper we introduce new distributions which are solutions of higher-order Laplace equations. It is proved that their densities can be obtained by folding and symmetrizing Cauchy distributions. Another class of probability laws related to higher-order Laplace equations is obtained by composing pseudo-processes with positively-skewed stable distributions which produce asymmetric Cauchy densities in the odd-order case. A special attention is devoted to the third-order Laplace equation where the connection between the Cauchy distribution and the Airy functions is obtained and analyzed.
\end{abstract}

\maketitle

\section{Introduction}
The Cauchy density 
\begin{equation}
p(x,t)=\frac{1}{\pi}\frac{t}{(x^2+t^2)} \label{Cauchylaw}
\end{equation}
solves the Laplace equation (see \citet{NANE08})
\begin{equation}
\frac{\partial^2 u}{\partial t^2} + \frac{\partial^2 u}{\partial x^2}=0, \quad  x \in \mathbb{R},\; t>0.
\end{equation}
The $n$-dimensional counterpart of \eqref{Cauchylaw}
\begin{equation}
p(\mathbf{x}, t) = \frac{\Gamma\left( \frac{n}{2} \right)}{\pi^{\frac{n}{2}}} \frac{t}{\left( t^2 + |\mathbf{x}|^2 \right)^{\frac{n}{2}}}, \quad \mathbf{x} \in \mathbb{R}^{n-1},\; t>0
\end{equation}
with characteristic function
\begin{equation}
\int_{\mathbb{R}^{n-1}} e^{i \langle \boldsymbol{\alpha}, \mathbf{x} \rangle} p(\mathbf{x}, t) d\mathbf{x} = \exp \left( -t |\boldsymbol{\alpha}| \right)
\end{equation}
solves the $n$-dimensional Laplace equation
\begin{equation}
\frac{\partial^2 p}{\partial t^2}  + \sum_{j=1}^{n-1} \frac{\partial^2 p}{\partial x_j^2}=0.
\end{equation}
The inspiring idea of this paper is to investigate the class of distributions which satisfy the higher-order Laplace equations of the form
\begin{equation}
\frac{\partial^n u}{\partial t^n} + \frac{\partial^n u}{\partial x^n}=0, \quad x \in \mathbb{R},\; t>0
\end{equation}
In a previous paper of ours we have shown that the law
\begin{equation}
p_4(x,t) = \frac{t}{\pi \sqrt{2}} \frac{x^2+t^2}{x^4+t^4} \label{pquattro}
\end{equation}
solves the fourth-order Laplace equation 
\begin{equation}
\frac{\partial^4 u}{\partial t^4} + \frac{\partial^4 u}{\partial x^4}=0, \quad x \in \mathbb{R},\; t>0.
\end{equation}
In Section 2 we analyze distributions related to equations of the form
\begin{equation}
\frac{\partial^{2^n} u}{\partial t^{2^n}} + \frac{\partial^{2^n} u}{\partial x^{2^n}}=0 \label{decapDO}
\end{equation}
which can be expressed in many alternative forms. The decoupling of the $2^n$-th order differential operator in \eqref{decapDO} 
\begin{equation*}
 \frac{\partial^{2^n}}{\partial t^{2^n}} + \frac{\partial^{2^n}}{\partial x^{2^n}}= \prod_{\begin{subarray}{c} k=-(2^{n-1} - 1)\\ k \; odd \end{subarray}}^{2^{n-1}-1} \left( \frac{\partial^2}{\partial t^2} + e^{i \frac{\pi k}{2^{n-1}}} \frac{\partial^2}{\partial x^2} \right)
\end{equation*}
suggests to represent distributions related to \eqref{decapDO} as
\begin{equation}
p_{2^n}(x,t) = \frac{1}{\pi 2^{n-1}} \sum_{\begin{subarray}{c} k=-(2^{n-1} - 1)\\ k \; odd \end{subarray}}^{2^{n-1}-1}  \frac{t\, e^{i\frac{\pi k}{2^n}}}{x^2 + ( t e^{i \frac{\pi k}{2^{n}}})^2}, \quad n \geq 2.
\label{pdueenne1}
\end{equation}
that is the superposition of Cauchy densities at imaginary times. Alternatively, we give a real-valued expression for \eqref{pdueenne1} as
\begin{equation}
p_{2^n}(x,t) = \frac{t (x^2 + t^2)}{2^{n-2} \pi (x^{2^n} + t^{2^n})} \sum_{\begin{subarray}{c} k=1 \\ k \textrm{ odd} \end{subarray}}^{2^{n-1}-1} \cos \frac{k\pi}{2^n} \prod_{\begin{subarray}{c} j=1,\,j \textrm{ odd}\\ j \neq k \end{subarray}}^{2^{n-1}-1}\left( x^4 + t^4 + 2x^2 t^2 \cos \frac{j\pi}{2^{n-1}} \right). \label{pdueenne2}
\end{equation}
The density \eqref{pdueenne2} can also be represented as
\begin{equation}
p_{2^n}(x, t) = \frac{t(x^2 +t^2 )}{2^{n-2}\pi} \sum_{\begin{subarray}{c} k=1 \\ k \textrm{ odd} \end{subarray}}^{2^{n-1}-1} \frac{\cos \frac{k \pi}{2^n}}{x^4 + t^4 + 2x^2 t^2 \cos \frac{k \pi}{2^{n-1}}}, \quad n \geq 2. \label{pdueenne3}
\end{equation}
Each component of the distribution \eqref{pdueenne3} is produced by folding and symmetrizing the density of the r.v.
$$V(t) = C\left( t \cos \frac{k\pi}{2^n} \right) - t \sin \frac{k \pi}{2^n}, \quad t>0,\; 1 \leq k \leq 2^{n-1}-1, \; k \textrm{ odd}$$
where $C(t)$, $t>0$ is the Cauchy symmetric process. The distributions \eqref{pdueenne3} differ from the Cauchy laws since they have a bimodal structure for all $n \geq 2$ as figures below show. For $n=2$, the distribution \eqref{pdueenne2} reduces to \eqref{pquattro} if we assume that the inner product appearing in formula \eqref{pdueenne2} is equal to one. Of course, the density  \eqref{pdueenne3} coincides with \eqref{pquattro} for $n=2$. For $n=3$ we get from \eqref{pdueenne2} and \eqref{pdueenne3} that
\begin{align}
p_{2^3}(x,t) = & \frac{t(x^2 + t^2)}{2\pi (x^8 + t^8)} \left[ (x^4 + t^4 - \sqrt{2} x^2 t^2)\cos \frac{\pi}{8} + (x^4 + t^4 + \sqrt{2}x^2 t^2)\sin \frac{\pi}{8} \right]\nonumber \\
= & \frac{t(x^2 + t^2)}{2 \pi} \left[ \frac{\sin \frac{\pi}{8}}{x^4 + t^4 - \sqrt{2}x^2 t^2} + \frac{\cos \frac{\pi}{8}}{x^4 + t^4 + \sqrt{2} x^2 t^2} \right]. \label{pdueenne4}
\end{align}
In \citet{DO3} we have shown that the density \eqref{pquattro} is the probability distribution of
\begin{equation*}
Q(t) = F(T_t), \quad t>0
\end{equation*}
where $F$ is the Fresnel pseudo-process and $T_t$, $t>0$ is the first passage time of a Brownian motion independent from $F$. The pseudo-process $F$ is constructed in \cite{DO3} by means of the fundamental solution (representing the density of the pseudo-process $F$) 
\begin{equation*}
u(x,t) = \frac{1}{\sqrt{4\pi t}} \cos \left( \frac{x^2}{2t} - \frac{\pi}{4} \right), \quad x\in \mathbb{R}, \; t>0
\end{equation*}  
of the equation of vibrations of rods
\begin{equation}
\frac{\partial^2 u}{\partial t^2} = - \frac{1}{2^2}\frac{\partial^4 u}{\partial x^4}. \label{eq-rods}
\end{equation}
The Fourier transform $U(\beta, t)$ of $u(x,t)$ is
$$U(\beta, t) = \cos \frac{\beta^2 t}{2}.$$ 
It is a new non-Markovian pseudo-process which can be analysed by means of the decoupling of \eqref{eq-rods} into two Schr\"{o}dinger equations. We note that
\begin{equation*}
\mathcal{Q}(t) = F(|B(t)|), \quad t>0
\end{equation*}
has density coinciding with the fundamental solution of the fourth-order heat equation 
\begin{equation*}
\frac{\partial u}{\partial t} = - \frac{\partial^4 u}{\partial x^4}.
\end{equation*}

We prove also that, for $k \in \mathbb{N}$, there are non-centered  Cauchy distributions which solve the equations
\begin{equation}
\frac{\partial^{2k+1} u}{\partial t^{2k+1}} + \frac{\partial^{2k+1}u}{\partial x^{2k+1}}=0.
\end{equation}
If $X_{2k+1}(t)$, $t>0$ is the pseudo-process whose density measure 
$$\mu_{2k+1}(dx, t) = \mu \{ X_{2k+1}(t) \in dx \}$$ 
solves the heat-type equations
\begin{equation}
\frac{\partial u}{\partial t} = - \frac{\partial^{2k+1} u}{\partial x^{2k+1}}, \quad k \in \mathbb{N} \label{eqHO}
\end{equation}
and $S_{\frac{1}{2k+1}}(t)$, $t>0$ is a positively skewed stable process  of order $\frac{1}{2k+1}$ we have that
\begin{equation}
Pr\{ X_{2k+1}(S_{\frac{1}{2k+1}}(t)) \in dx \}/dx = \frac{t\cos\frac{\pi}{2(2k+1)}}{\pi \left[ \left(x + (-1)^{k+1} t\sin \frac{\pi}{2(2k+1)}\right)^2 + t^2 \cos^2 \frac{\pi}{2(2k+1)}\right]}. \label{eq118}
\end{equation}
Pseudo-processes are constructed by attributing to cylinders 
\begin{equation*}
C = \left\lbrace x : \, a_j \leq x(t_j) \leq b_j , \; j =1,2,\ldots , n \right\rbrace
\end{equation*}
of sample paths $x: t \mapsto x(t)$ the following signed measure
\begin{equation*}
\mu(C) = \int_{a_1}^{b_1} \cdots \int_{a_n}^{b_n} \prod_{j=1}^n p_n(x_j-x_{j-1}; t_j - t_{j-1})dx_j
\end{equation*}
and therefore has the following explicit form
\begin{equation*}
\mu_{2k+1}(x, t) = \frac{1}{2\pi} \int_{-\infty}^{+\infty} e^{-i \xi x - t(-i\xi )^{2n+1} }d\xi .
\end{equation*}
The extension of the measure $\mu(C)$ to the field generated by the cylinders $C$ is described in \cite{Hoc78} for the fourth-order pseudo-process and can be adopted also in the odd-order case treated here (see also \cite{kry60, LCH03, Lad63, Ors91}). A review of higher-order equations appearing in different areas of applied sciences can be found in \cite{beghin-et-all-2007}.

We show below that the densities \eqref{eq118} solve also the following second-order p.d.e.
\begin{equation*}
\frac{\partial^2 u}{\partial t^2} + \frac{\partial^2 u}{\partial x^2} = 2 \sin \frac{\pi}{2(2k+1)} \frac{\partial^2 u}{\partial t\, \partial x}.
\end{equation*}
We have investigated in detail the case of third-order Laplace equation 
\begin{equation}
\frac{\partial^3 u}{\partial t^3} + \frac{\partial^3 u}{\partial x^3}=0
\end{equation}
and have shown that
\begin{align}
Pr\{ X_3(S_{\frac{1}{3}}(t)) \in dx \} = & dx \int_0^\infty \frac{1}{\sqrt[3]{3s}} Ai\left( \frac{x}{\sqrt[3]{3s}} \right) \frac{t}{s} \frac{1}{\sqrt[3]{3s}} Ai\left( \frac{t}{\sqrt[3]{3s}} \right) ds  \\
= & dx \frac{\sqrt{3}}{2} t \frac{x-t}{x^3 - t^3} = dx\frac{\sqrt{3}}{2} \frac{t}{x^2 + xt + t^2}\nonumber  \\
= & dx \frac{t\cos \frac{\pi}{6}}{(x + t \sin \frac{\pi}{6})^2 + t^2 \cos^2\frac{\pi}{6}}. \nonumber
\end{align}
The pictures of the Cauchy distributions \eqref{eq118} show that the location parameter $ t \sin \frac{\pi}{2(2k+1)}$ tends to zero as $k \to \infty$ while the scale parameter tends to one, $t \cos \frac{\pi}{2(2k+1)} \to t$. This means that the asymmetry of the Cauchy densities decreases as $k$ increases and is maximal for $k=1$. The decrease of parameters of \eqref{eq118} (with $k$ increasing) is due to the growing symmetrization of the fundamental solutions of equations \eqref{eqHO}.

By suitably combining the distribution \eqref{eq118} for $k=1$, we arrive at the density 
\begin{equation}
p_6(x,t) = \frac{\sqrt{3}}{2\pi} t \frac{(x^2 + t^2)\cos \frac{\pi}{6} + xt}{\left(x^2 + t^2 + xt \cos \frac{\pi}{6}\right)^2  - 3 x^2 t^2 \sin^2\frac{\pi}{6}} \label{eq6intro}
\end{equation}
which solves the equation
\begin{equation}
\frac{\partial^6 u}{\partial t^6} + \frac{\partial^6 u}{\partial x^6}=0. 
\end{equation}
The probability density \eqref{eq6intro} displays the unimodal structure of the Cauchy distribution.

\section{Hyper-Cauchy distributions}
In this section we analyze the distribution related to Laplace-type equations of the form 
\begin{equation}
\left( \frac{\partial^{2^n}}{\partial t^{2^n}} + \frac{\partial^{2^n}}{\partial x^{2^n}} \right) u=0, \quad n>1.
\end{equation}
For $n \geq 2$ we obtain a new class of distributions having the form
\begin{equation}
p_{2^n}(x, t) = \frac{t(x^2 + t^2)}{2^{n-2}\pi (x^{2^n} + t^{2^n})} g(x, t), \quad x \in \mathbb{R}, \; t>0 \label{classF}
\end{equation}
where $g(x, t)$ is a polynomial of order $2^n - 2^2$. For $n=2$, formula \eqref{classF} yields the distribution
\begin{equation}
p_4(x, t) = \frac{t(x^2 + t^2)}{\sqrt{2} \pi (x^4 + t^4)}, \quad x \in \mathbb{R}, \; t>0
\end{equation}
emerging in the analysis of Fresnel pseudo-processes (see \citet{DO3}). 

The main result of this section is given in the next theorem.
\begin{te}
The hyper-Cauchy density 
\begin{align}
p_{2^n}(x, t) = & \frac{1}{\pi 2^{n-1}} \sum_{\begin{subarray}{c} k= - (2^{n-1}-1)\\ k \textrm{ odd} \end{subarray}}^{2^{n-1}-1} \frac{t e^{i \frac{\pi k}{2^n}}}{x^2 + (t e^{i \frac{\pi k}{2^n}})^2}\label{eqWuno}
\end{align}
solves the equation
\begin{equation}
\left( \frac{\partial^{2^n}}{\partial t^{2^n}} + \frac{\partial^{2^n}}{\partial x^{2^n}} \right) u =0, \quad x \in \mathbb{R},\; t>0, \quad n >1. \label{eqWeq}
\end{equation}
A real-valued expression of \eqref{eqWuno} reads
\begin{align}
p_{2^n}(x,t) = & \frac{t (x^2 + t^2)}{\pi 2^{n-2} (x^{2^n}+ t^{2^n})} \sum_{\begin{subarray}{c} k= 1\\ k \textrm{ odd} \end{subarray}}^{2^{n-1}-1} \cos \frac{k\pi}{2^n} \prod_{\begin{subarray}{c} k \neq j = 1\\j  \textrm{ odd} \end{subarray}}^{2^{n-1}-1} \left( x^4 + t^4 + 2x^2 t^2 \cos \frac{j\pi}{2^{n-1}} \right) \label{eqWdue}
\end{align}
or equivalently
\begin{equation}
p_{2^n}(x,t) = \frac{t(x^2 + t^2)}{2^{n-2} \pi} \sum_{\begin{subarray}{c} k=1\\ \textrm{k odd} \end{subarray}}^{2^{n-1}-1} \frac{\cos \frac{k\pi}{2^{n}}}{x^4+t^4 + 2x^2t^2\cos \frac{k\pi}{2^{n-1}}}, \quad \textrm{for }\; n >1. \label{eqWtre}
\end{equation}
\end{te}
\begin{proof}
In order to check that \eqref{eqWuno} satisfies equation \eqref{eqWeq} we resort to Fourier transforms
$$U(\beta, t) = \int_{-\infty}^{+\infty} e^{i \beta x} u(x, t) dx.$$
Equation \eqref{eqWeq} becomes
\begin{equation}
\frac{\partial^{2^n}U}{\partial t^{2^n}} + (-i \beta)^{2^n} U = \frac{\partial^{2^n}U}{\partial t^{2^n}} + \beta^{2^n} U = 0. \label{furEq}
\end{equation}
The solutions of the algebraic equation associated to \eqref{furEq} have the form 
\begin{equation}
r_j = |\beta | \, e^{i \pi \frac{2j+1}{2^n}}, \quad 0 \leq j \leq 2^{n} - 1.
\end{equation}
In order to construct bounded solutions to \eqref{furEq} we restrict ourselves to
\begin{equation}
U(\beta, t) = \frac{1}{2^{n-1}} \sum_{\begin{subarray}{c} k=-(2^{n-1} - 1)\\ k \; odd \end{subarray}}^{2^{n-1}-1} e^{- t |\beta | e^{i \frac{k\pi}{2^n}}} \label{eqWeqFur}
\end{equation}
where the normalizing constant in \eqref{eqWeqFur} is chosen equal to $1/2^{n-1}$ so that $U(\beta , 0) = 1$. The inverse of \eqref{eqWeqFur} is \eqref{eqWuno}. We check directly that each term of \eqref{eqWuno} has Fourier transform solving equation \eqref{furEq}. For all odd values of $k$, we have that 
\begin{align*}
& \int_{-\infty}^{+\infty} e^{i\beta x} \left( \frac{\partial^{2^n}}{\partial t^{2^n}} + \frac{\partial^{2^n}}{\partial x^{2^n}} \right) \frac{1}{\pi} \left( \frac{t e^{i \frac{k\pi}{2^n}}}{x^2 + (t e^{i \frac{k\pi}{2^n}})^2} \right) dx\\
= & \frac{\partial^{2^n}}{\partial t^{2^n}} e^{-t |\beta | e^{i \frac{k\pi}{2^n}} } + (-i\beta)^{2^n} e^{-t |\beta | e^{i \frac{k\pi}{2^n}}}\\
= & \left( \beta^{2^n} e^{ik\pi} + i^{2^n} \beta^{2^n} \right) e^{-t |\beta | e^{i \frac{k\pi}{2^n}} }\\
= & \left( (-1)^k \beta^{2^n} + \beta^{2^n} \right) e^{-t |\beta | e^{i \frac{k\pi}{2^n}} } = 0
\end{align*}
because $k$ is odd. In order to obtain \eqref{eqWdue} we observe that, in view of \eqref{eqWuno} we can write
\begin{align*}
p_{2^n}(x, t) = & \frac{1}{\pi} \sum_{\begin{subarray}{c} k= -(2^{n-1}-1)\\ k \textrm{ odd} \end{subarray}}^{2^{n-1}-1} \frac{c_k\, t^{|2k-1|} x^{2^n - |2k-1|-1}}{ \prod_{\begin{subarray}{c} k=-(2^{n-1}-1)\\k \textrm{ odd} \end{subarray}}^{2^{n-1}-1} (x^2 + (t e^{i \frac{k\pi}{2^n}})^2)}
\end{align*}
where
\begin{align}
 \prod_{\begin{subarray}{c} k=-(2^{n-1}-1)\\k \textrm{ odd} \end{subarray}}^{2^{n-1}-1} (x^2 + (t e^{i \frac{k\pi}{2^n}})^2) = x^{2^n} + t^{2^n} \label{eqWsei}
\end{align}
and $c_k$ are constants evaluated below. Result \eqref{eqWsei} can be obtained directly by solving the equation $x^{2^n} + t^{2^n}=0$ or by successively regrouping the terms of the right-hand side of \eqref{eqWsei}. We have at first that
\begin{align}
\prod_{\begin{subarray}{c} k=-(2^{n-1}-1)\\ k \textrm{ odd} \end{subarray}}^{2^{n-1}-1} (x^2 + (t e^{i \frac{k\pi}{2^n}})^2) = & \prod_{k=1, \, k \textrm{ odd}}^{2^{n-1}-1} \left(x^4 + t^4 + 2x^2 t^2 \cos \frac{k\pi}{2^{n-1}}\right)\nonumber \\
= & \prod_{k=1, \, k \textrm{ odd}}^{2^{n-2}-1} \left(x^8 + t^8 + 2 x^4 t^4 \cos \frac{k\pi}{2^{n-2}}\right)\nonumber \\
= & \cdots \nonumber \\
= & \left(x^{2^n} + t^{2^n} + 2 x^2 t^2 \cos \frac{\pi}{2}\right)\nonumber \\
= & x^{2^n} + t^{2^n}. \nonumber 
\end{align}
In view of \eqref{eqWsei} we can rewrite \eqref{eqWuno} as
\begin{align*}
p_{2^n}(x, t) = & \frac{t}{\pi 2^{n-1}(x^{2^n} + t^{2^n})} \sum_{\begin{subarray}{c} k=-(2^{n-1}-1)\\ k \textrm{ odd} \end{subarray}}^{2^{n-1}-1} \prod_{\begin{subarray}{c} j =-(2^{n-1}-1)\\ j \textrm{ odd,}\, j \neq k \end{subarray}}^{2^{n-1}-1} (x^2 + (te^{i \frac{\pi j}{2^n}})^2) e^{i \frac{\pi k}{2^n}}
\end{align*}
where
\begin{align*}
& \sum_{\begin{subarray}{c} k=-(2^{n-1}-1)\\ k \textrm{ odd} \end{subarray}}^{2^{n-1}-1} \prod_{\begin{subarray}{c} j =-(2^{n-1}-1)\\ j \textrm{ odd,}\, j \neq k \end{subarray}}^{2^{n-1}-1} (x^2 + (te^{i \frac{\pi j}{2^n}})^2) e^{i \frac{\pi k}{2^n}}\\
= & \sum_{\begin{subarray}{c} k=-(2^{n-1}-1)\\ k \textrm{ odd} \end{subarray}}^{2^{n-1}-1} \prod_{\begin{subarray}{c} j =1\\ j \textrm{ odd,}\, j \neq k \end{subarray}}^{2^{n-1}-1} (x^4 + t^4 + 2x^2 t^2 \cos \frac{\pi j}{2^{n-1}})  (x^2 + (te^{-i \frac{\pi k}{2^n}})^2) e^{i \frac{\pi k}{2^n}}\\
= & \sum_{\begin{subarray}{c} k=1\\ k \textrm{ odd} \end{subarray}}^{2^{n-1}-1} \prod_{\begin{subarray}{c} j =1\\ j \textrm{ odd,}\, j \neq k \end{subarray}}^{2^{n-1}-1} (x^4 + t^4 + 2x^2 t^2 \cos \frac{\pi j}{2^{n-1}})(x^2 e^{i \frac{k\pi}{2^n}} + t^2 e^{-i \frac{k\pi}{2^n}}+ x^2 e^{-i \frac{k\pi}{2^n}} + t^2 e^{i \frac{k\pi}{2^n}})\\
= & 2 (x^2 + t^2) \sum_{\begin{subarray}{c} k=1\\ k \textrm{ odd} \end{subarray}}^{2^{n-1}-1} \cos \frac{k \pi}{2^n} \prod_{\begin{subarray}{c} j =1\\ j \textrm{ odd,}\, j \neq k \end{subarray}}^{2^{n-1}-1} (x^4 + t^4 + 2x^2 t^2 \cos \frac{\pi j}{2^{n-1}})
\end{align*}
and thus
\begin{align*}
p_{2^n}(x, t) = & \frac{t(x^2+t^2)}{\pi 2^{n-2}(x^{2^n} + t^{2^n})} \sum_{\begin{subarray}{c} k=1\\ k \textrm{ odd} \end{subarray}}^{2^{n-1}-1} \cos \frac{k \pi}{2^n} \prod_{\begin{subarray}{c} j =1\\ j \textrm{ odd,}\, j \neq k \end{subarray}}^{2^{n-1}-1} (x^4 + t^4 + 2x^2 t^2 \cos \frac{\pi j}{2^{n-1}}).
\end{align*}
Furthermore, from the fact that
\begin{equation*}
x^{2^n} + t^{2^n} = \prod_{k=1, \, k \textrm{ odd}}^{2^{n-2}-1} \left(x^8 + t^8 + 2 x^4 t^4 \cos \frac{k\pi}{2^{n-2}}\right)
\end{equation*}
we obtain that
\begin{equation*}
p_{2^n}(x, t) = \frac{t(x^2 + t^2)}{2^{n-2} \pi} \sum_{\begin{subarray}{c} k=1\\ k \textrm{ odd} \end{subarray}}^{2^{n-1}-1}  \frac{\cos  \frac{k\pi}{2^n}}{x^4 + t^4 + 2x^2 t^2 \cos \frac{k\pi}{2^{n-1}}}.
\end{equation*}
\end{proof}

\begin{os}
\normalfont
In order to prove that the density \eqref{eqWdue} integrates to unity we present the following calculation
\begin{align}
\int_{-\infty}^{+\infty} \frac{x^2 + t^2}{x^4+t^4 + 2x^2t^2\cos \frac{k\pi}{2^{n-1}}} dx = & 2 \int_{0}^{+\infty} \frac{x^2 + t^2}{x^4 + t^4 + 2x^2t^2\cos \frac{k\pi}{2^{n-1}}} dx \notag \\ 
= & \frac{2}{t}\int_{0}^{+\infty} \frac{y^2 +1}{y^4 + 1 + 2y\cos \frac{k\pi}{2^{n-1}}}dy \notag \\
= & \frac{2}{t} \int_0^{\frac{\pi}{2}} \frac{1}{\tan^4 \theta + 1 + 2 \tan^2 \theta \cos \frac{k\pi}{2^{n-1}}} \frac{d\theta}{\cos^4 \theta}\notag \\
= & \frac{2}{t} \int_0^{\frac{\pi}{2}} \frac{d\theta}{\sin^4 \theta + \cos^4 \theta + 2 \sin^2 \theta \cos^2\theta \cos \frac{k\pi}{2^{n-1}}}\notag \\
= & \frac{2}{t}\int_0^{\frac{\pi}{2}} \frac{d\theta}{1 - \frac{1-\cos\frac{k\pi}{2^{n-1}}}{2} \sin^2 2\theta}\notag  \\
= & \frac{2}{t} \int_0^{\frac{\pi}{2}} \frac{d\theta}{1- \frac{1}{2}\left(1- \cos \frac{k\pi}{2^{n-1}} \right)\left( \frac{1-\cos 4\theta}{2}\right)}\notag \\
= & \frac{1}{2t} \int_0^{2\pi} \frac{d\phi}{1- \frac{1- \cos \frac{k\pi}{2^{n-1}}}{4} + \frac{1}{4} \left( 1- \cos \frac{k\pi}{2^{n-1}} \right) \cos \phi}\notag \\
= & \frac{2}{t} \int_0^{2\pi} \frac{d\phi}{\left(3+ \cos \frac{k\pi}{2^{n-1}}\right) + \left( 1-\cos\frac{k\pi}{2^{n-1}}\right) \cos \phi}\notag \\
= & \frac{2}{t} \frac{2\pi}{\sqrt{\left( 3 + \cos \frac{k\pi}{2^{n-1}} \right)^2 - \left( 1 -\cos\frac{k\pi}{2^{n-1}}\right)^2}} \notag \\
= & \frac{\pi \sqrt{2}}{t} \frac{1}{\sqrt{1 + \cos \frac{k\pi}{2^{n-1}}}}\notag \\
= & \frac{\pi}{t} \frac{1}{\cos \frac{k\pi}{2^n}} \label{unity}.
\end{align} 
From \eqref{eqWtre}, in view of \eqref{unity}, we can conclude that
$$\int_{-\infty}^{+\infty} p_{2^n}(x, t)\, dx = 1$$
\end{os}

\begin{os}
\normalfont
From \eqref{eqWuno}, for $n=2$ we obtain that
$$ p_4(x,t)=\frac{1}{2\pi}\left[ \frac{t e^{i\frac{\pi}{4}}}{x^2 + (t e^{i\frac{\pi}{4}})^2} +  \frac{t e^{-i\frac{\pi}{4}}}{x^2 + (t e^{-i\frac{\pi}{4}})^2} \right] $$ 
with Fourier transform
$$\int_{-\infty}^{+\infty} e^{i \beta x} p_4(x, t)dx = e^{-\frac{t}{\sqrt{2}}|\beta |} \cos \frac{\beta t}{\sqrt{2}}.$$
From \eqref{eqWdue} and \eqref{eqWtre} we have that
\begin{equation}
p_4(x, t) = \frac{t}{\sqrt{2}\pi} \frac{x^2+t^2}{x^4 + t^4}. \label{p4law}
\end{equation}
The law \eqref{p4law} has two maxima as Figure \ref{fig3} shows.
\end{os}

\begin{os}
\normalfont
For $n=3$, from \eqref{eqWdue}, we have that
$$p_8(x, t) =  \frac{t (x^2 + t^2)}{2 \pi (x^8 + t^8)} \Bigg[ \left(x^4 +t^4 +2x^2t^2\cos\frac{\pi}{4} \right) \cos \frac{3\pi}{8} + \left(x^4 + t^4 +2x^2t^2\cos\frac{3\pi}{4} \right) \cos \frac{\pi}{8} \Bigg].$$
From the fact that
$$\cos\frac{3\pi}{4} = -\cos \frac{\pi}{4} \quad \textrm{and} \quad \cos \frac{3\pi}{8} = \sin \frac{\pi}{8}$$
we write
\begin{align*}
p_8(x,t) = & \frac{t (x^2 + t^2)}{2 \pi (x^8 + t^8)} \Bigg[ & \left(x^4 +t^4 +\sqrt{2}x^2t^2 \right) \sin \frac{\pi}{8} + \left(x^4 + t^4 -\sqrt{2}x^2t^2 \right) \cos \frac{\pi}{8} \Bigg].
\end{align*}
From \eqref{eqWtre} we have also that
\begin{align}
p_8(x,t) = & \frac{t}{2\pi}\left[ \frac{x^2+t^2}{x^4+t^4-\sqrt{2}x^2t^2}\sin \frac{\pi}{8} +  \frac{x^2+t^2}{x^4+t^4 + \sqrt{2}x^2t^2}\cos \frac{\pi}{8} \right] . \label{p8law}
\end{align}
From \eqref{eqWuno} we obtain the characteristic function
\begin{equation*}
\int_{\mathbb{R}} e^{i \beta x} p_{8}(x, t)dx = \frac{1}{2^2} \left[ e^{-t |\beta | \cos\frac{\pi}{8}} \cos \left( t \beta \sin \frac{\pi}{8} \right) + e^{-t |\beta | \sin \frac{\pi}{8}} \cos \left( t \beta \cos \frac{\pi}{8} \right) \right].
\end{equation*}
The density $p_8(x, t)$ is a bimodal curve as well as $p_4(x, t)$. The maxima of $p_8(x, t)$ are heigher than those of $p_4(x,t)$ as Figure \ref{fig3} shows. Also $p_{2^n}(x, t)$ displays a bimodal structure with   the height peaks increasing as $n$ increases. The form of $p_{2^n}(x,t)$ reminds the structure of densities of fractional diffusions governed by equations
$$\frac{\partial^\nu u}{\partial t^\nu} = \lambda^2 \frac{\partial^2 u}{\partial x^2}$$ 
for $1 < \nu < 2$ (see \cite{OB09}).
\end{os}

\begin{figure}
\caption{The profile of the functions $p_4$ (dotted line), formula \eqref{p4law} and $p_8$, formula \eqref{p8law}.}
\includegraphics[scale=.5]{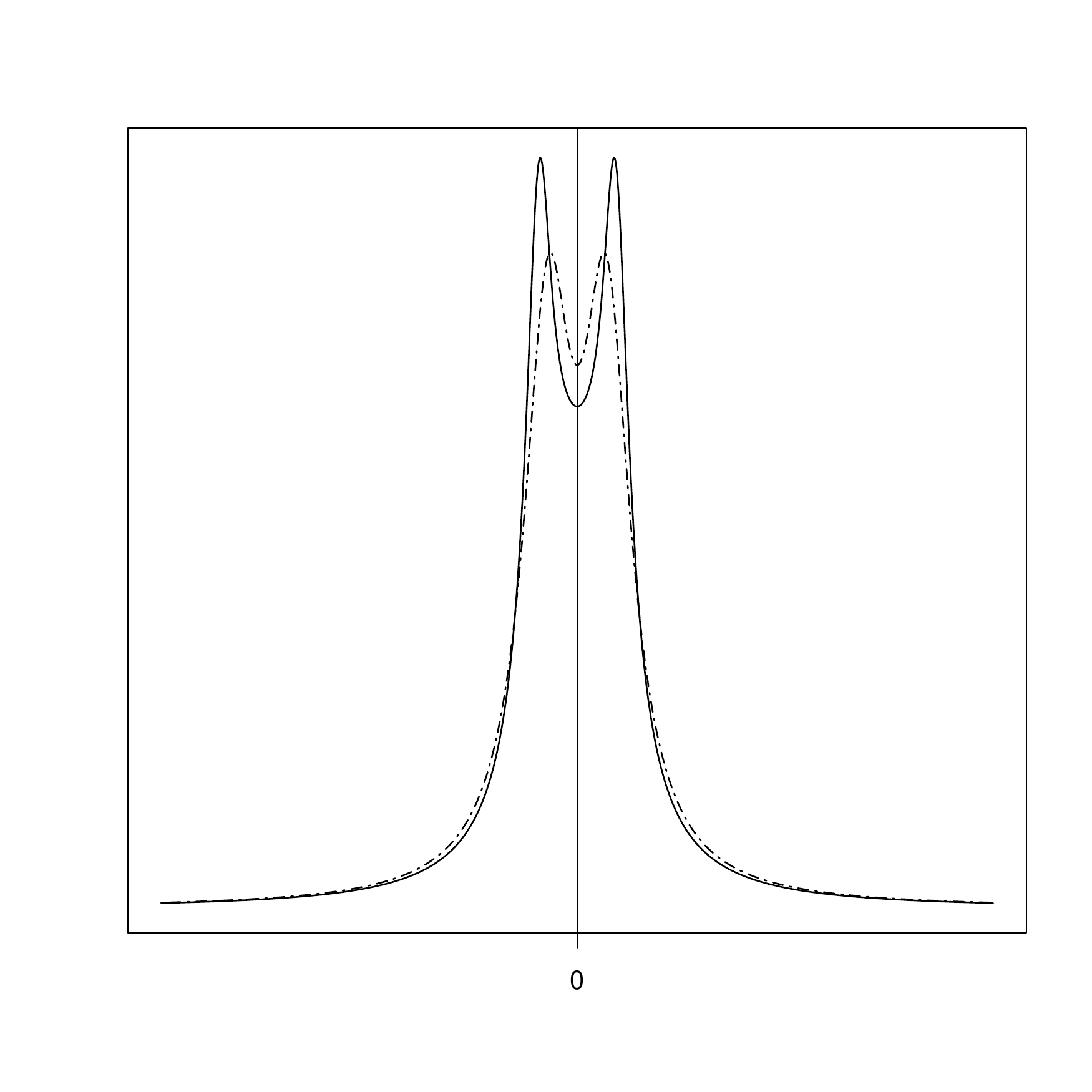} 
\label{fig3}
\end{figure}

\begin{os}
\normalfont
The result \eqref{eqWtre} can conveniently be rewritten as
\begin{equation}
p_{2^n}(x, t) = \frac{t}{\pi(x^2 + t^2)} \left[ \frac{1}{2^{n-1}} \sum_{\begin{subarray}{c} k =1\\ k \textrm{ odd} \end{subarray}}^{2^{n-1}-1} \frac{x^4 + t^4 + 2x^2t^2}{x^4 + t^4 + 2x^2t^2\cos \frac{k\pi}{2^{n-1}}} \cos\frac{k\pi}{2^n} \right]. \label{212eq}
\end{equation}
The factor in square parenthesis measures, in some sense, the disturbance of $p_{2^n}$ on the classical Cauchy. For $n=2$, we have in particular that 
\begin{equation}
p_{2^2}(x, t) = \frac{t}{\pi(x^2 + t^2)}\frac{1}{\sqrt{2}} \left[ 1 + \frac{2x^2t^2}{x^4 + t^4} \right] = \frac{t}{\sqrt{2} \pi} \frac{x^2+t^2}{x^4 + t^4}. \label{densWsette}
\end{equation}
The density \eqref{densWsette} has two symmetric maxima at $x=\pm t \sqrt{\sqrt{2}-1}$ and a minimum at $x=0$ (see Fig. 6 of \citet{DO3}). The terms
\begin{equation}
g_k(x, t) = \frac{x^4 + t^4 + 2x^2 t^2}{x^4 + t^4 + 2x^2 t^2 \cos \frac{k\pi}{2^{n-1}}}
\end{equation}
display two maxima at $x = \pm t$ with height depending on $k$ and whose profile is depicted in Figure \ref{fig1}.
\end{os}

\begin{figure}
  \caption{The profile of the function $g_k$ for $n=3$ and $k=1$ (dotted line), $k=3$.}
    \includegraphics[width=0.6\textwidth]{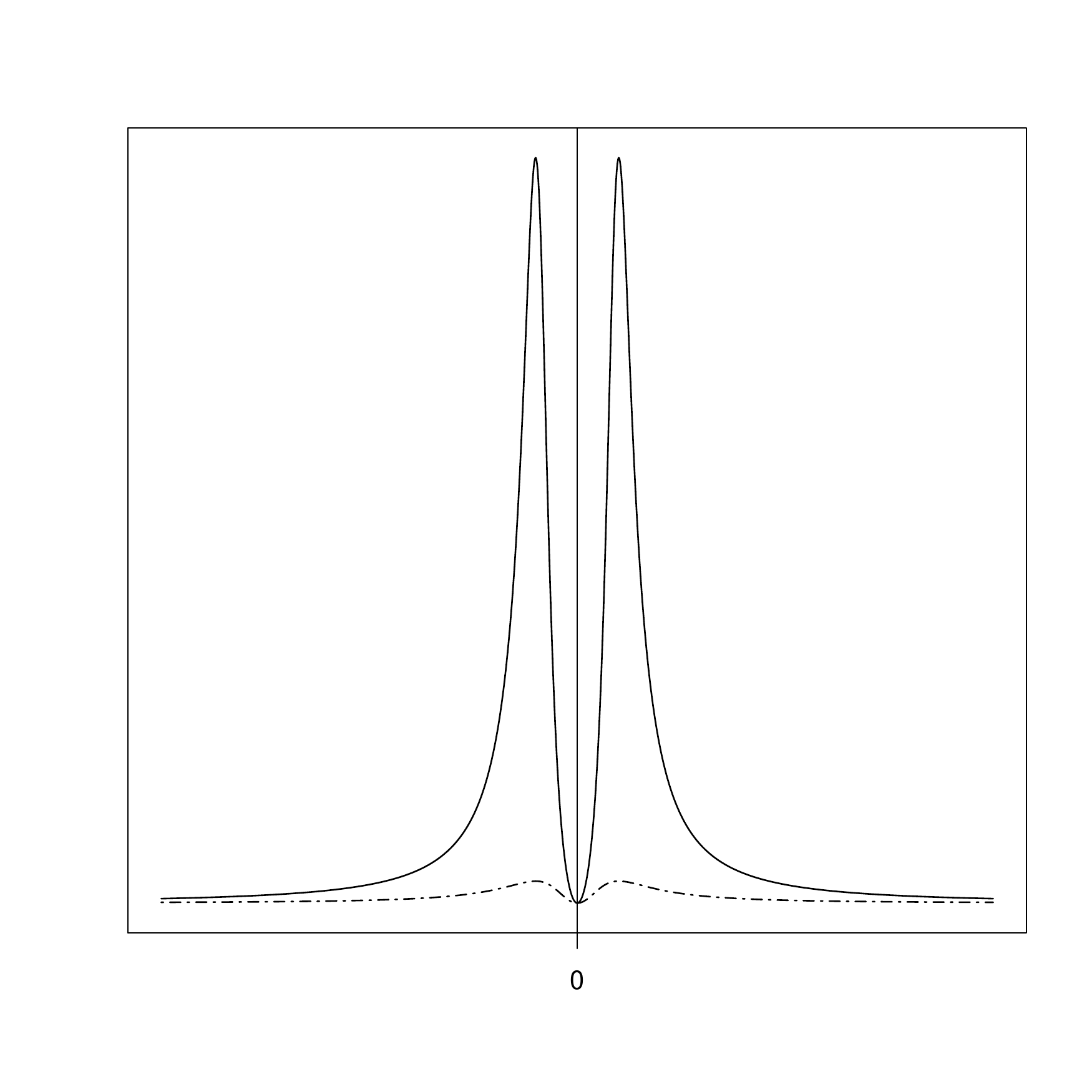}
    \label{fig1}
\end{figure}

\begin{os}
\normalfont
The density $p_{2^n}(x, t)$ can be written as
\begin{equation}
p_{2^n}(x,t) = \frac{t (x^2 + t^2)}{2^{n-2}\pi (x^{2^n} + t^{2^n})} Q(x, t)\label{3sette}
\end{equation}
where $Q(x, t)$ is a polynomial of order $2^n - 2^2$. For $n=2$ the function $Q(x, t)$ reduces to $\cos \frac{\pi}{4}$. For $n=3$, 
\begin{align*}
Q(x, t) = & (x^4 + t^4 + \sqrt{2}x^2t^2) \sin \frac{\pi}{8} + (x^4 + t^4 - \sqrt{2} x^2 t^2)\cos \frac{\pi}{8}.
\end{align*}
The expression \eqref{3sette} shows that the probability law $p_{2^n}(x, t)$, $x \in \mathbb{R}$, $t>0$ shares with the classical Cauchy density the property of non-existence of the mean value.
\end{os}

\begin{figure}
\caption{The profile of the functions $p_{2^n}$, formula \eqref{eqWtre}, for $n=5,10,15,20$.}
\includegraphics[scale=.6]{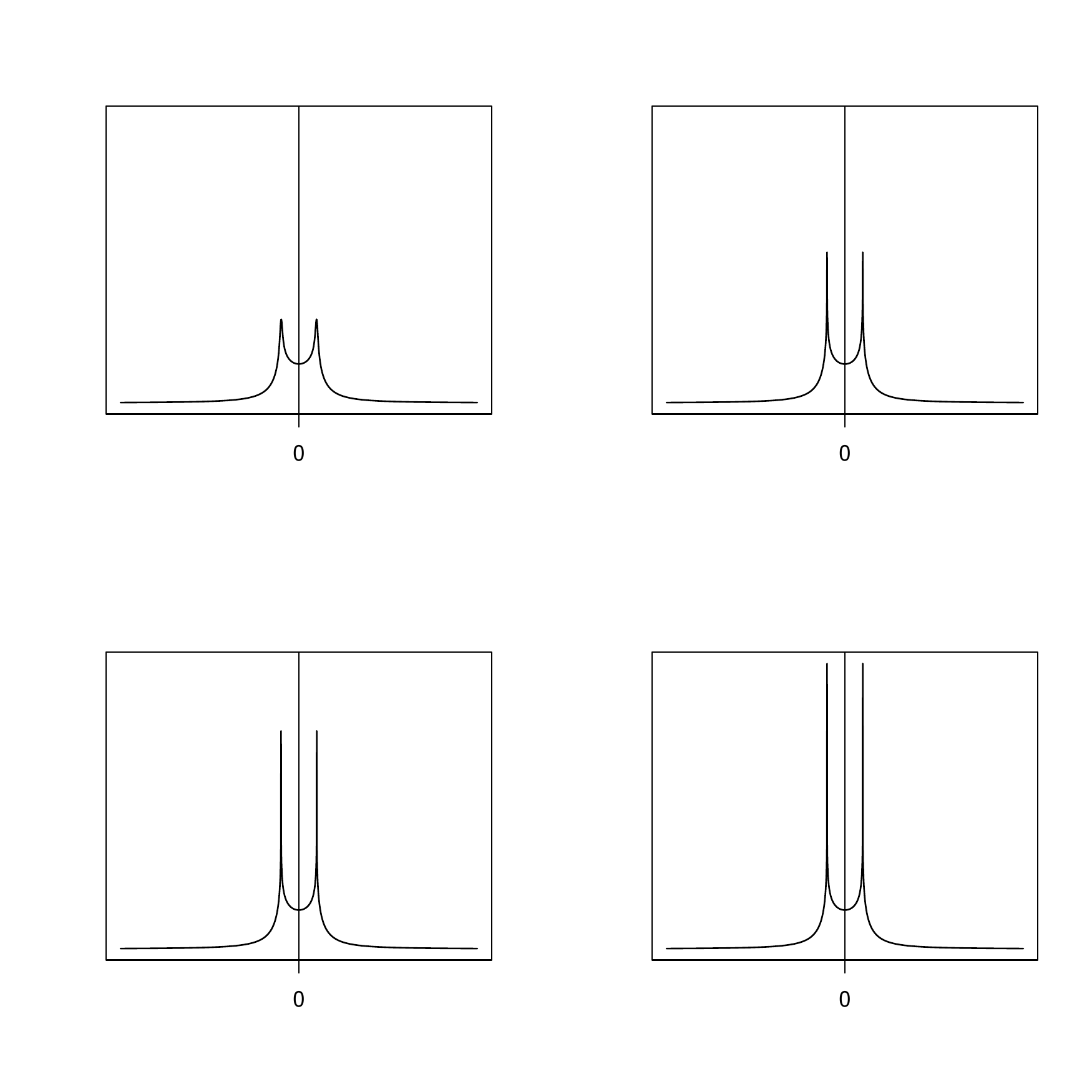} 
\label{fig4}
\end{figure}

\begin{os}
\normalfont
The density of the hyper-Cauchy can also be presented in an alternative form by regrouping the terms in the right-hand side of \eqref{eqWtre} as
\begin{align}
& \sum_{\begin{subarray}{c} k=1\\ \textrm{k odd} \end{subarray}}^{2^{n-2}-1} \left[ \frac{\sin \frac{k\pi}{2^n}}{x^4 + t^4 - 2 x^2 t^2 \cos \frac{k\pi}{2^{n-1}}} + \frac{\cos \frac{k\pi}{2^n}}{x^4+t^4 + 2x^2t^2 \cos \frac{k\pi}{2^{n-1}}} \right]\label{eqsum216}\\
= & \sum_{\begin{subarray}{c} k=1\\ \textrm{k odd} \end{subarray}}^{2^{n-2}-1} \frac{\left(x^4 +t^4 + 2x^2t^2\cos\frac{k\pi}{2^{n-1}}\right)\sin \frac{k \pi}{2^n} + \left( x^4 + t^4 - 2x^2 t^2 \cos \frac{k \pi}{2^{n-1}} \right) \cos \frac{k\pi}{2^n}}{x^8 + t^8 - 2 x^4 t^4 \cos \frac{k\pi}{2^{n-2}}}.\notag
\end{align}
For $ n=3$, from \eqref{eqsum216}, we get again that
$$p_{8}(x,t) = \frac{t(x^2 + t^2 )}{2\pi (x^8 + t^8 )}\left[ \left(x^4 +t^4 + \sqrt{2}x^2t^2\right)\sin \frac{k \pi}{8} + \left( x^4 + t^4 - \sqrt{2}x^2 t^2  \right) \cos \frac{k\pi}{8} \right].$$
\end{os}

\begin{os}
\normalfont
The r.v. 
\begin{equation}
W(t) = \bigg| C\left(t \cos \frac{\pi k}{2^{n}}\right)  - t \sin \frac{\pi k}{2^{n}} \bigg| \label{symC}
\end{equation}
(where $C(t)$, $t>0$ is the Cauchy process) has probability density 
\begin{equation}
f_k(w, t) = \frac{2t(w^2 + t^2) \cos \frac{k\pi}{2^n}}{\pi (w^4 + t^4 + 2w^2t^2 \cos \frac{k\pi}{2^{n-1}})}, \quad w>0. \label{eq218}
\end{equation}
Indeed, we have that
\begin{align}
Pr\left\lbrace  W(t) < w \right\rbrace = & \int_{-w + t \sin \frac{k\pi}{2^n}}^{+w + t \sin \frac{k\pi}{2^n}} dy \frac{t \cos \frac{k\pi}{2^n}}{\pi (y^2 + t^2 \cos^2 \frac{k\pi}{2^n})} \label{lawCtras}
\end{align}
and
\begin{align*}
f_k(w, t) = & \frac{d}{dw}Pr\left\lbrace  \bigg| C\left(t \cos \frac{\pi k}{2^{n}}\right)  - t \sin \frac{\pi k}{2^{n}} \bigg| < w \right\rbrace\\
= & \frac{t \cos \frac{k\pi}{2^n}}{\pi \bigg((w + t \sin \frac{k\pi}{2^n})^2 + t^2 \cos^2 \frac{k\pi}{2^n}\bigg)} + \frac{t \cos \frac{k\pi}{2^n}}{\pi \bigg((-w + t \sin \frac{k\pi}{2^n})^2 + t^2 \cos^2 \frac{k\pi}{2^n}\bigg)}\\
= & \frac{t \cos \frac{k\pi}{2^n}}{\pi \bigg(w^2 + 2 wt \sin \frac{k\pi}{2^n}+ t^2 \bigg)} + \frac{t \cos \frac{k\pi}{2^n}}{\pi \bigg(w^2 - 2 wt \sin \frac{k\pi}{2^n}+ t^2 \bigg)}\\
= & \frac{2t (w^2 + t^2) \cos \frac{k\pi}{2^n}}{\pi (w^2 + t^2 + 2wt\sin \frac{k\pi}{2^{n}}) (w^2 + t^2 - 2wt\sin \frac{k\pi}{2^{n}})}\\
= & \frac{2t(w^2 + t^2) \cos \frac{k\pi}{2^n}}{\pi (w^4 + t^4 + 2w^2t^2 \cos \frac{k\pi}{2^{n-1}})}
\end{align*}
because
$$2\sin^2 \frac{k\pi}{2^n} = 1 - \cos \frac{k\pi}{2^{n-1}}.$$
By symmetrizing \eqref{symC} as follows
$$Z(t) = \frac{W_1(t) - W_2(t)}{2}$$
where $W_1(t), W_2(t)$ are independent copies of $W(t)$ we obtain a distribution of the form
\begin{equation}
h_k(w, t) = \frac{t (w^2 + t^2) \cos \frac{k\pi}{2^n}}{\pi (w^4 + t^4 + 2w^2t^2 \cos \frac{k\pi}{2^{n-1}})}, \quad w\in \mathbb{R} \label{eq220}
\end{equation}
which coincides with each term of \eqref{212eq}. This construction explains the reason for which each term in \eqref{212eq} has two symmetric maxima at $w=\pm t \sqrt{2\sin \frac{k\pi}{2^n} - 1}$ for $k :\, \sin \frac{\pi k}{2^n} > \frac{1}{2}$.
\end{os}

\begin{figure}
\caption{The figure shows how the distribution \eqref{eq220} can be constructed from the Cauchy density by folding and symmetrizing, in the cases $n=3$, $k=1$ (top figures) and $k=3$ (bottom figures). The dotted line gives the density of the folded distribution \eqref{eq218}.}
\includegraphics[scale=.5]{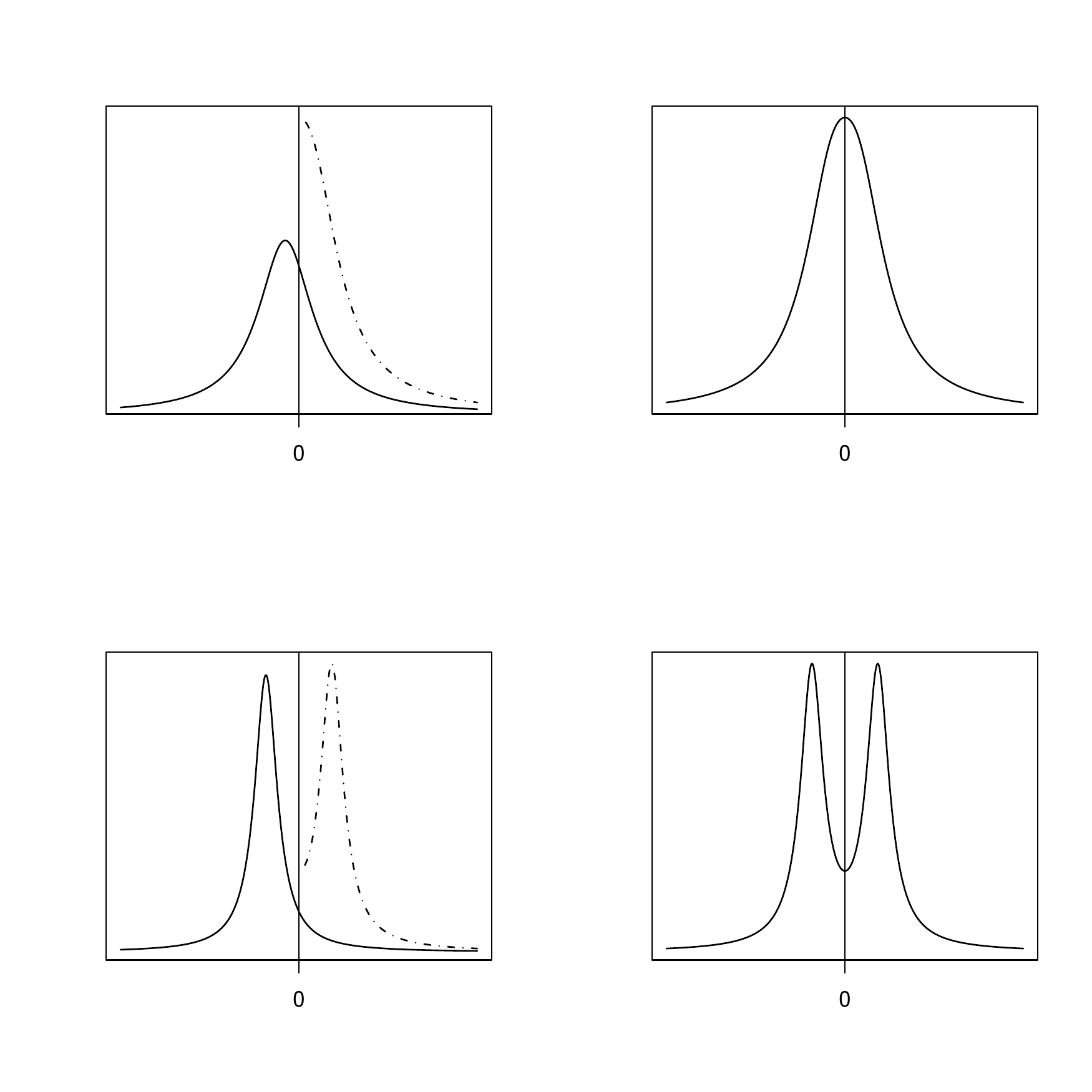} 
\end{figure}

\section{Higher-order Laplace-type equation}
Let us consider the pseudo-processes related to higher-order heat-type equations 
\begin{equation}
\frac{\partial u}{\partial t} = c_n \frac{\partial^n u}{\partial x^n}, \quad x \in \mathbb{R},\; t>0, \quad n >2. \label{HigLapEq}
\end{equation}
where $c_n=(-1)^{\frac{n}{2}+1}$ for $n$ even and $c_n=\pm 1$ for $n$ odd. 

Pseudo-processes constructed by exploiting the sign-varying measures obtained as fundamental solutions to \eqref{HigLapEq} have been examined in many papers since the beginning of the Sixties. A description of the procedure of construction of pseudo-processes can be found, for example in \citet{kry60}, \citet{Lad63}, \citet{Hoc78}, \citet{Ors91}, \citet{LCH03}. In the case where $n=2k+1$, $c_{2k+1}=-1$, the fundamental solution to \eqref{HigLapEq} reads
\begin{equation}
u_{2k+1}(x, t) = \frac{1}{2\pi} \int_{-\infty}^{+\infty} e^{-i \beta x + i (-1)^k t \beta ^{2k+1}}d\beta.
\end{equation}
In particular, for $k=1$
\begin{equation}
u_3(x, t) = \frac{1}{\pi} \int_0^\infty \cos\left( \beta x + \beta^3 t \right) d\beta = \frac{1}{\sqrt[3]{3t}}Ai\left( \frac{x}{\sqrt[3]{3t}} \right)
\end{equation}
where
\begin{equation*}
Ai(x) = \frac{\sqrt{x}}{3}\left[ I_{-\frac{1}{3}}\left( \frac{2}{3} x^{3/2} \right) - I_{\frac{1}{3}}\left( \frac{2}{3}x^{3/2}\right) \right]
\end{equation*}
is the Airy function (see for example \citet{LE}). 

In this section we study the composition of pseudo-processes with stable processes $S_\alpha(t)$, $t>0$, $\alpha \in (0,1)$ whose characteristic function reads
\begin{align}
\mathbb{E}e^{i\beta S_\alpha(t)} = \exp\left(-t|\beta |^{\alpha} e^{- i \frac{\pi \gamma}{2} \frac{\beta}{|\beta |}} \right) = \exp\left( - \sigma \, t |\beta |^{\alpha} \left( 1 -  i \theta \frac{\beta}{|\beta |} \tan\frac{\pi \alpha}{2} \right) \right)
\label{charstabledens}
\end{align}
where $\sigma=\cos \pi \gamma /2>0$ and 
\begin{equation*}
\theta = \cot\left( \frac{\pi \alpha}{2} \right) \tan\left( \frac{\pi \gamma}{2} \right).
\end{equation*}
The parameter $\gamma$ must be chosen in such a way that $\theta \in [-1,1]$ for $\alpha \in (0,1)$. The skewness parameter $\theta=1$ (that is $\gamma=\alpha$) corresponds to positively skewed stable distributions. For the density
\begin{equation*}
p_{\alpha}(x, \gamma, t) = \frac{1}{2\pi} \int_{-\infty}^{+\infty} e^{- i \beta x} \mathbb{E} e^{i \beta S_\alpha(t)} d\beta
\end{equation*}
we have the scaling property
\begin{equation}
p_{\alpha}(x, \gamma, t) = \frac{1}{t^{1/\alpha}} p_{\alpha}\left( \frac{x}{t^{1/\alpha}}, \gamma, 1\right). \label{selfsimilar}
\end{equation}
For $\alpha \in (0,1)$, we have the series representation of stable density (see \cite[page 245]{OB09})
\begin{equation}
p_{\alpha}(x;\gamma, 1) = \frac{\alpha}{\pi} \sum_{r=0}^{\infty} (-1)^r \frac{\Gamma(\alpha(r+1))}{r!} x^{-\alpha(r+1)-1} \sin\left( \frac{\pi}{2}(\gamma + \alpha)(r+1) \right) . \label{serieslawS}
\end{equation}
\begin{te}
The composition of the pseudo-process $X_{2k+1}(t)$, $t>0$ with the stable process $S_{\frac{1}{2k+1}}(t)$, $t>0$, $k \in \mathbb{N}$, has a Cauchy probability distribution which can be written as
\begin{equation}
Pr\{ X_{2k+1}(S_{\frac{1}{2k+1}}(t) \in dx \}/dx =\frac{t\, \cos \frac{\pi}{2(2k+1)}}{\pi \left[ \left( x + (-1)^{k+1}  t\, \sin \frac{\pi}{2(2k+1)} \right)^2 + t^2\, \cos^2 \frac{\pi}{2(2k+1)} \right]}
\label{densCauchAsym}
\end{equation}
with $x \in \mathbb{R},\; t>0$. The density function \eqref{densCauchAsym} is a solution to the higher-order Laplace equation
\begin{equation}
\frac{\partial^{2k+1} u}{\partial t^{2k+1}} + \frac{\partial^{2k+1} u}{\partial x^{2k+1}}=0, \quad x \in \mathbb{R},\; t>0\label{h-oEQ}
\end{equation}
\end{te}
\begin{proof}
For $\theta=1$, $\alpha=\gamma=1/2k+1$, in view of \eqref{charstabledens} we have that
\begin{align}
U(\beta, t) & = \int_{-\infty}^{+\infty} e^{i \beta x}Pr\{ X_{2k+1}(S_{\frac{1}{2k+1}}(t)) \in dx \}\notag \\
& =  \int_0^\infty Pr\{S_{\frac{1}{2k+1}}(t) \in ds\} \int_{-\infty}^{+\infty} e^{i \beta x} u_{2k+1}(x, s) \, dx \notag \\
& = \int_{0}^{\infty} e^{i s (-1)^k \beta^{2k+1}} Pr\{ S_{\frac{1}{2k+1}}(t) \in ds \}\notag \\
& =  \exp\left( - t \Big| (-1)^k \beta^{2k+1} \Big|^{\frac{1}{2k+1}} \cos \frac{\pi}{2 (2k+1)} \left( 1 - i  \textrm{ sgn }\Big((-1)^k \beta^{2k+1}\Big) \tan \frac{\pi }{2 (2k+1)} \right) \right) \notag \\
& = \exp\left( - t | \beta | \left( \cos \frac{\pi}{2 (2k+1)} - i (-1)^k \frac{\beta}{|\beta |} \sin \frac{\pi}{2(2k+1)}\right) \right) \notag \\
& = \exp\left( - t | \beta | \cos \frac{\pi}{2 (2k+1)} - i (-1)^k t \beta \sin \frac{\pi}{2(2k+1)}\right) . \label{furChat}
\end{align}
This is the characteristic function of a Cauchy distribution with scale parameter $t\cos \frac{\pi}{2(2k+1)}$ and location parameter $t (-1)^{k+1} \sin \frac{\pi}{2(2k+1)}$. Formula \eqref{furChat} can also be rewritten as
\begin{align}
U(\beta, t) & = \exp\left( - t | \beta | \left( \cos \frac{\pi}{2 (2k+1)} - i (-1)^k \frac{\beta}{|\beta |} \sin \frac{\pi}{2(2k+1)}\right) \right)\notag\\
& = \exp\left( - t | \beta | \left( \cos \left(\frac{\pi}{2 (2k+1)} (-1)^k \frac{\beta}{|\beta |} \right) - i \sin \left(\frac{\pi}{2 (2k+1)} (-1)^k \frac{\beta}{|\beta |}\right) \right) \right)\notag\\
& =  \exp\left( - t |\beta | e^{-i \frac{\pi}{2(2k+1)} (-1)^k \frac{\beta}{|\beta |}} \right). \label{furchat2}
\end{align}
The Fourier transform of equation \eqref{h-oEQ} becomes 
\begin{equation}
\frac{\partial^{2k+1}U}{\partial t^{2k+1}} + (-i \beta)^{2k+1}U=0.
\end{equation}
The derivative of order $2k+1$ of \eqref{furchat2} is
\begin{equation}
\frac{\partial^{2k+1} U}{\partial t^{2k+1}} (\beta , t) = (-|\beta |)^{2k+1} \left( e^{-i \frac{\pi}{2(2k+1)} (-1)^k \frac{\beta}{|\beta |}} \right)^{2k+1}\, U(\beta , t)
\end{equation}
and this shows that the Cauchy distribution \eqref{densCauchAsym} solves the higher-order Laplace equation \eqref{h-oEQ}.
\end{proof}

\begin{os}
\normalfont
We notice that
\begin{align}
\int_0^\infty Pr\{ X_{2k+1}(S_{\frac{1}{2k+1}}(t)) \in dx \} = & \frac{1}{\pi} \int_{(-1)^{k+1} \tan \frac{\pi}{2(2k+1)}}^{\infty} \frac{dy}{1+y^2} \notag \\
= & \frac{1}{2}\left(1 + \frac{(-1)^k}{2k+1} \right) \label{eq218}
\end{align}
which is somehow in accord with \citet{LCH03}. The results \eqref{densCauchAsym} and \eqref{eq218} show that the mode of the Cauchy law \eqref{densCauchAsym} approaches the origin as $k$ increases. 
\end{os}

Let us consider the process of the form $X_3(S_{\frac{1}{3}}(t))$, $t>0$ where $X_3$ is a pseudo-process whose measure density is governed by the third-order heat equation
\begin{equation}
\frac{\partial u}{\partial t} = - \frac{\partial^3 u}{\partial x^3}, \quad x \in \mathbb{R},\; t>0
\end{equation} 
and $S_{\frac{1}{3}}$ is the stable process of order $1/3$. The distribution of $X_3(S_{\frac{1}{3}}(t))$, $t>0$  reads
\begin{equation}
Pr\{ X_3(S_{\frac{1}{3}}(t)) \in dx \} = dx \int_0^\infty \frac{1}{\sqrt[3]{3s}} Ai\left( \frac{x}{\sqrt[3]{3s}} \right) \frac{t}{s} \frac{1}{\sqrt[3]{3s}} Ai\left( \frac{t}{\sqrt[3]{3s}} \right)\, ds \label{lawXS3}
\end{equation}
where
\begin{equation}
Pr\{ S_{\frac{1}{3}}(t) \in ds \} = ds \frac{t}{s} \frac{1}{\sqrt[3]{3s}} Ai\left( \frac{t}{\sqrt[3]{3s}} \right), \quad s \geq 0,\; t>0 \label{lawS3}
\end{equation}
for which
\begin{align*}
\int_0^\infty Pr\{ S_{\frac{1}{3}}(t) \in ds \} = & \int_{0}^{\infty} ds \frac{t}{s} \frac{1}{\sqrt[3]{3s}} Ai\left( \frac{t}{\sqrt[3]{3s}} \right)\\
= & (w=t/\sqrt[3]{3s}) = 3 \int_0^\infty Ai(w)\, dw = 1.
\end{align*}

\begin{coro}
The law \eqref{lawXS3} solves the higher-order Laplace equation
\begin{equation}
\frac{\partial^3 u}{\partial t^3} + \frac{\partial^3 u}{\partial x^3} = 0, \quad x \in \mathbb{R},\; t>0
\end{equation}
and can be written as
\begin{align}
Pr\{ X_3(S_{\frac{1}{3}}(t)) \in dx \} = & \frac{dx}{\pi} \frac{\frac{\sqrt{3}}{2} t}{\left( x +  \frac{t}{2} \right)^2 + \frac{3t^2}{4}}\label{3ordlaw}\\
= & \frac{dx}{\pi} \frac{3^{1/2}}{2} \frac{t}{x^2 + xt + t^2}\notag \\
= & dx \frac{3^{1/2} \, t}{2 \pi} \frac{x - t}{x^3 - t^3}.\notag
\end{align}
\label{coroAiry}
\end{coro}
\begin{proof}
The Fourier transform of \eqref{lawXS3} becomes
\begin{equation}
\int_{-\infty}^{\infty} e^{i\beta x} Pr\{X_3(S_{\frac{1}{3}}(t)) \in dx \} = \int_0^\infty e^{-i\beta^3 s } Pr\{ S_{\frac{1}{3}}(t) \in ds \}. \label{intFourier}
\end{equation}
We show that \eqref{lawS3} is a stable law of order $1/3$. In view of the representation of the the Airy function $(4.10)$ of \citet{OB09}
\begin{equation}
Ai(w) = \frac{3^{-2/3}}{\pi} \sum_{k=0}^\infty \frac{(3^{1/3} w)^k }{k!} \sin\left(  \frac{2\pi}{3} (k+1) \right) \Gamma\left( \frac{k+1}{3} \right)
\end{equation}
we can write that
\begin{align*}
\frac{t}{s} \frac{1}{\sqrt[3]{3s}} Ai\left( \frac{t}{\sqrt[3]{3s}} \right) = & \frac{t}{3\pi s \sqrt[3]{s}} \sum_{k=0}^{\infty} \left( \frac{t}{\sqrt[3]{s}} \right)^{k} \frac{1}{k!} \sin\left( \frac{2\pi}{3}(k+1) \right) \Gamma\left( \frac{k+1}{3} \right)
\end{align*}
We consider the series expansion \eqref{serieslawS} of the stable density (with $t=1$)  for which \eqref{charstabledens} holds true. For $\alpha=\gamma=1/3$ (that is $\theta=+1$), $x=s/t^3$ in \eqref{serieslawS} we get that
\begin{align*}
p_{\frac{1}{3}}\left(\frac{s}{t^3}; \frac{1}{3}, 1 \right) = & \frac{1}{3\pi} \sum_{k=0}^{\infty} \frac{(-1)^k}{k!} \left(  \frac{s}{t^3} \right)^{-\frac{k+1}{3} - 1} \sin\left( \frac{\pi}{3}(k+1) \right) \Gamma\left( \frac{k+1}{3} \right)\\
= & (\textrm{by } 4.5 \textrm{ of \cite{OB09}} )\\
= & \frac{1}{3\pi} \frac{t^4}{s\sqrt[3]{s}} \sum_{k=0}^{\infty} \left( \frac{t}{\sqrt[3]{s}} \right)^k \frac{1}{k!} \sin\left( \frac{2\pi}{3}(k+1) \right) \Gamma\left( \frac{k+1}{3} \right)\\
= & t^3 \left[ \frac{t}{s} \frac{1}{\sqrt[3]{3s}} Ai\left( \frac{t}{\sqrt[3]{3s}} \right) \right]
\end{align*}
and thus, from \eqref{selfsimilar}, we have that
\begin{equation*}
\frac{1}{t^3} p_{\frac{1}{3}}\left(\frac{s}{t^3}; \frac{1}{3}, 1 \right) = p_{\frac{1}{3}}\left( s; \frac{1}{3}, t \right) = \frac{t}{s} \frac{1}{\sqrt[3]{3s}} Ai\left( \frac{t}{\sqrt[3]{3s}} \right), \quad s,t>0.
\end{equation*}
We now evaluate the integral \eqref{intFourier}. We have that 
\begin{align}
 & \int_0^\infty e^{-i\beta^3 s } Pr\{ S_{\frac{1}{3}}(t) \in ds \} \nonumber \\
= & \exp\left( - \cos \frac{\pi}{6} \, t | -\beta^3 |^{\frac{1}{3}} \left( 1 -  i \textrm{ sgn } (-\beta^3) \tan\frac{\pi}{6} \right) \right)\nonumber \\
= & \exp\left( - \frac{\sqrt{3}}{2} t |\beta | \left( 1 +  i \textrm{ sgn } (\beta) \frac{1}{\sqrt{3}} \right) \right)\nonumber \\
= & \exp\left( - \frac{\sqrt{3}}{2} t |\beta | -  i \frac{t}{2} \beta \right) \label{eq10}
\end{align}
since $\textrm{sgn }(-\beta^3) = \textrm{sgn }(-\beta)= -\textrm{sgn }(\beta) = -\frac{\beta}{|\beta |}$. From \eqref{eq10} we infer that
\begin{align}
Pr\{ X_3(S_{\frac{1}{3}}(t)) \in dx \} = & \frac{dx}{2\pi} \int_{-\infty}^{+\infty} e^{- i \beta x} \exp\left( - \frac{\sqrt{3}}{2} t |\beta | -  i \frac{t}{2} \beta \right) \, d\beta \label{acc11}\\
= & \frac{dx}{\pi} \frac{\frac{\sqrt{3}}{2} t}{\left( x +  \frac{t}{2} \right)^2 + \frac{3t^2}{4}} = \frac{dx}{\pi} \frac{3^{1/2}}{2} \frac{t}{x^2 + xt + t^2}\nonumber \\ 
= & dx \frac{3^{1/2} \, t}{2 \pi} \frac{x - t}{x^3 - t^3} \nonumber
\end{align}
\end{proof}

\begin{os}
\normalfont
We observe that the r.v. $X_3(S_{\frac{1}{3}}(t))$ possesses Cauchy distribution with scale parameter $\sqrt{3}t/2$ and location parameter $-t/2$. Furthermore, it solves the third-order Laplace-type equation
\begin{equation}
\frac{\partial^3 u}{\partial t^3} + \frac{\partial^3 u}{\partial x^3} = 0.
\end{equation}
\end{os}

\begin{os}
\normalfont
From the fact that
\begin{equation}
\frac{1}{\sqrt[3]{3t}} Ai\left( \frac{x}{\sqrt[3]{3t}} \right) = \frac{1}{3\pi} \sqrt{\frac{x}{t}} K_{1/3}\left( \frac{2}{3^{3/2}} \frac{x^{3/2}}{\sqrt{t}} \right), \quad x,t>0
\end{equation}
we can write, for $x >0$,
\begin{align}
Pr\{ X_3(S_{\frac{1}{3}}(t)) \in dx \}/dx = & \int_{0}^{\infty}  \frac{1}{3\pi} \sqrt{\frac{x}{s}} K_{1/3}\left( \frac{2}{3^{3/2}} \frac{x^{3/2}}{\sqrt{s}} \right) \, \frac{t}{s} \frac{1}{3\pi} \sqrt{\frac{t}{s}} K_{1/3}\left( \frac{2}{3^{3/2}} \frac{t^{3/2}}{\sqrt{s}} \right)\, ds\\
= & \frac{2\sqrt{xt^3}}{3^2\pi^2} \int_{0}^\infty s \, K_{1/3}\left( \frac{2 x^{3/2}}{3^{3/2}} s \right)\, K_{1/3}\left( \frac{2t^{3/2}}{3^{3/2}} s \right)\, ds.
\end{align}
In view of (see \cite[formula 6.521]{GR})
\begin{equation*}
\int_0^\infty s \, K_{\nu}(ys)\,K_{\nu}(zs)\, ds = \frac{\pi (yz)^{-\nu}(y^{2\nu} - z^{2\nu})}{2\sin \pi \nu\, (y^2 - z^2)}, \quad \Re\{ y+z\}>0,\; |\Re\{ \nu \}|<1
\end{equation*}
we get that
\begin{equation}
Pr\{ X_3(S_{\frac{1}{3}}(t)) \in dx \} = dx\,\frac{3^{1/2} \, t}{2 \pi} \frac{x - t}{x^3 - t^3}, \quad x,t>0 \label{renorm3ord}
\end{equation}
which coincides with \eqref{3ordlaw}.
\end{os}

\begin{figure}[h!]
  \caption{The profile of the function \eqref{3ordlaw}.}
  \centering
    \includegraphics[width=0.4\textwidth]{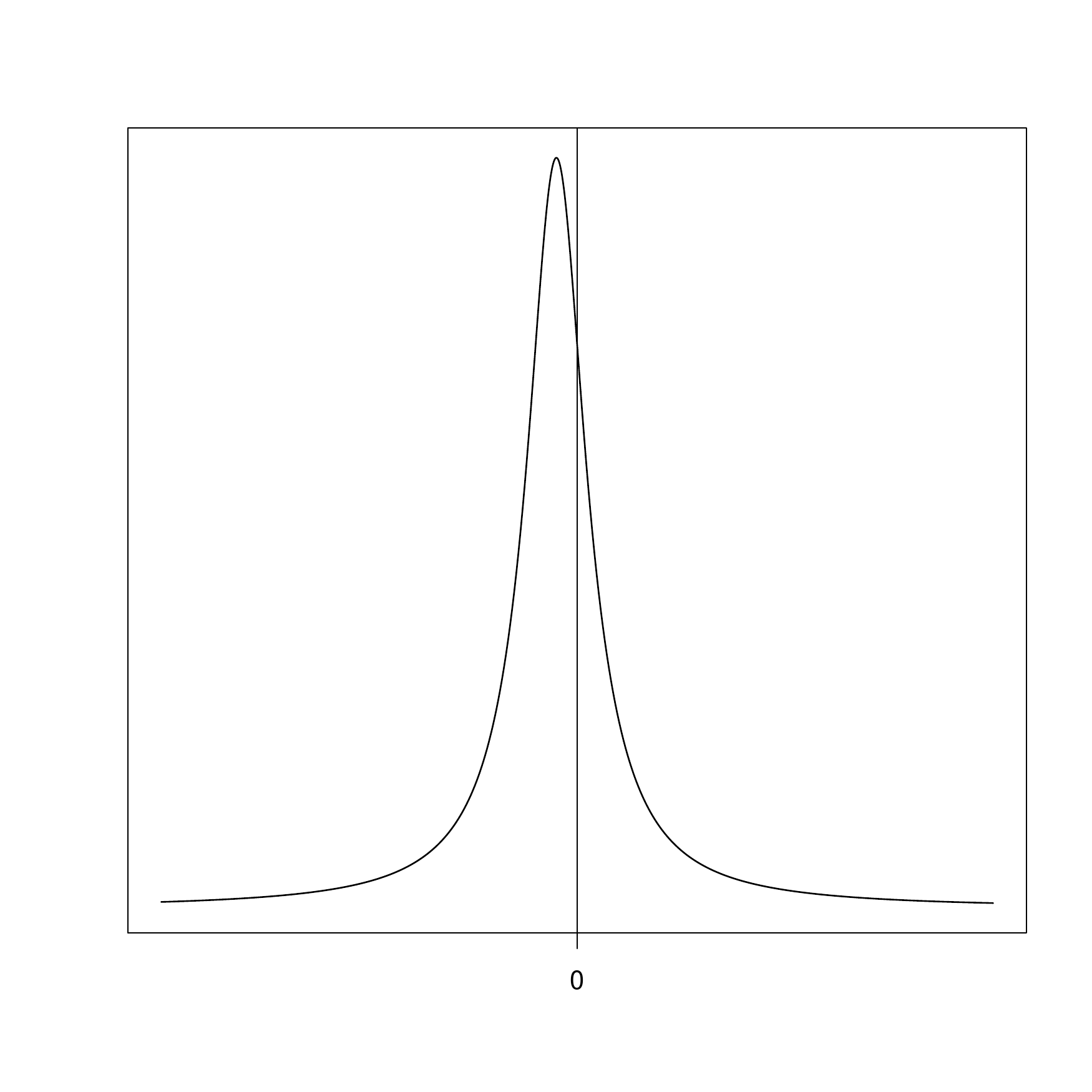}
\end{figure}

The Cauchy densities pertaining to the composition $X_{\frac{1}{2k+1}}(S_{\frac{1}{2k+1}}(t))$, $t>0$, solve also a second-order p.d.e. as we show in the next theorem.
\begin{te}
The Cauchy densities 
\begin{equation}
f(x,t;m) = \frac{1}{\pi} \frac{t \cos \frac{\pi}{2m}}{(x + t \sin \frac{\pi}{2m})^2 + t^2 \cos^2 \frac{\pi}{2m}}, \quad m \in \mathbb{N}
\label{osforE1}
\end{equation}
satisfy the following second-order equation
\begin{equation}
\frac{\partial^2 f}{\partial t^2} + \frac{\partial^2 f}{\partial x^2} = 2 \sin \frac{\pi}{2m} \frac{\partial^2 f}{\partial x \partial t}, \quad x \in \mathbb{R}, \; t>0.\label{oseqE1}
\end{equation}
\end{te}
\begin{proof}
It is convenient to write \eqref{osforE1} as a composed function
$$f(u, v) = \frac{1}{\pi} \frac{u}{u^2 + v^2}$$
where
$$u= t \cos \frac{\pi}{2m}, \qquad v= x + t\sin \frac{\pi}{2m}.$$
Since
\begin{align*}
& \frac{\partial f}{\partial t} = \cos \frac{\pi}{2m}\frac{\partial f}{\partial u} + \sin \frac{\pi}{2m} \frac{\partial f}{\partial v}\\
& \frac{\partial^2 f}{\partial t^2} = \cos^2 \frac{\pi}{2m} \frac{\partial^2 f}{\partial u^2} + 2 \cos \frac{\pi}{2m} \sin \frac{\pi}{2m} \frac{\partial^2 f}{\partial u \partial v} + \sin^2 \frac{\pi}{2m} \frac{\partial^2 f}{\partial v^2}\\
& \frac{\partial f}{\partial x} = \frac{\partial f}{\partial v} \quad \textrm{ and } \quad \frac{\partial^2 f}{\partial x^2} = \frac{\partial^2 f}{\partial v^2}
\end{align*}
and
$$\frac{\partial^2 f}{\partial u^2} + \frac{\partial^2 f}{\partial v^2} = 0$$
we have that
\begin{align*}
\frac{\partial^2 f}{\partial t^2} + \frac{\partial^2 f}{\partial x^2} = & \cos^2 \frac{\pi}{2m} \frac{\partial^2 f}{\partial u^2} + \frac{\partial^2 f}{\partial v^2} + 2\sin \frac{\pi}{2m} \cos \frac{\pi}{2m} \frac{\partial^2 f}{\partial u \partial v} + \sin^2 \frac{\pi}{2m} \frac{\partial^2 f}{\partial v^2}\\
= & \frac{\partial^2 f}{\partial v^2} \left[ 1- \cos^2 \frac{\pi}{2m} + \sin^2 \frac{\pi}{2m} \right] + 2 \sin \frac{\pi}{2m} \cos \frac{\pi}{2m} \frac{\partial^2 f}{\partial u \partial v}\\
= & 2 \sin \frac{\pi}{2m} \frac{\partial}{\partial v} \left[ \sin \frac{\pi}{2m} \frac{\partial f}{\partial v} + \cos \frac{\pi}{2m} \frac{\partial f}{\partial u} \right]\\
= & 2 \sin \frac{\pi}{2m} \frac{\partial}{\partial x} \frac{\partial f}{\partial t}
\end{align*}
\end{proof}

\begin{os}
\normalfont
The characteristic function of \eqref{osforE1} is
$$\int_{-\infty}^{+\infty} e^{i\beta x} f(x, t; m) dx = e^{-t |\beta | \cos \frac{\pi}{2m} - i \beta t \sin \frac{\pi}{2m}}$$
and can be obtained by considering the bounded solution to the Fourier transform  of \eqref{oseqE1}
$$\frac{d^2 F}{dt^2} + 2i \beta \sin \frac{\pi}{2m} \frac{dF}{dt} - \beta^2 F = 0. $$
\end{os}

For the even-order Laplace equations we have the following result.
\begin{te}
The solution to the higher-order Laplace-type equation
\begin{equation}
\frac{\partial^{2n} u}{\partial t^{2n}} = (-1)^n \frac{\partial^{2n} u}{\partial x^{2n}}, \quad x \in \mathbb{R},\; t>0\label{h-oEQ-even}
\end{equation}
subject to the initial conditions
\begin{equation}
\begin{cases}
u(x, 0)=\delta(x)\\
\frac{\partial^k u}{\partial t^k}(x, t)\Big|_{t=0^+} =  \frac{(-1)^k \, k!}{\pi |x|^{k+1}}  \cos \frac{\pi (k+1)}{2}, \quad 0< k < 2n
\end{cases}
\label{iconPP}
\end{equation}
is the classical Cauchy distribution given by
\begin{equation}
u(x, t) =Pr\{ X_{2n}(S_{\frac{1}{2n}}(t)) \in dx \}/dx = \frac{t}{\pi (x^2 + t^2)}, \quad x \in \mathbb{R},\; t>0 
\end{equation}
where $X_{2n}(t)$, $t>0$ is a pseudo-process such that
$$\mathbb{E} e^{i \beta X_{2n}(t)} = e^{-t \beta^{2n}}.$$
\end{te}
\begin{proof}
The pseudo-process $X_{2n}(t)$, $t>0$ related to the equation
$$\frac{\partial u}{\partial t} = (-1)^{n+1}\frac{\partial^{2n} u}{\partial t^{2n}} $$
has fundamental solution whose Fourier transform reads
$$\int_{-\infty}^{+\infty} e^{i \beta x} u(x, t)dx = e^{-t \beta^{2n}}.$$
If $S_{\frac{1}{2n}}(t)$, $t>0$ is a stable subordinator with Laplace transform
\begin{equation}
\mathbb{E} \exp \left(-\lambda S_{\frac{1}{2n}}(t) \right) = \exp \left( -t \lambda^{\frac{1}{2n}}\right), \quad \lambda >0, \; t>0 \label{inview219}
\end{equation}
the characteristic function of $X_{2n}(S_\frac{1}{2n}(t))$, $t>0$ becomes
\begin{align}
\int_{-\infty}^{+\infty} e^{i\beta x} Pr\{ X_{2n}(S_{\frac{1}{2n}}(t)) \in dx \} = & \int_0^\infty e^{- s \beta^{2n} } Pr\{ S_{\frac{1}{2n}}(t) \in ds \} \notag \\
= & \exp\left( - t |\beta | e^{i \frac{\pi r}{n}} \right), \quad r=0,1,\ldots , 2n-1\label{fuNOk}
\end{align}
For $r=0$, we have the characteristic function of the Cauchy symmetric law. For $r \neq 0$ and $n \leq r \leq 2n-1$ we have a function which is not absolutely integrable and, for $0<r<n-1$ is not a characteristic function (but can be regarded as a Cauchy r.v. at a complex time). The functions
$$F_r(\beta , t) = e^{- t |\beta | e^{i\frac{\pi r}{2^n}}}$$ 
for all $0 \leq r \leq 2n-1$ are solutions to
$$\frac{\partial^{2n}F_r}{\partial t^{2n}} = (-1)^{n+1}F_r.$$
We now check that for $0\leq k \leq 2n-1$ the initial conditions \eqref{iconPP} are verified by the Cauchy distribution. Indeed, 
\begin{align*}
\frac{\partial^k u}{\partial t^{k}}(x, t)\Big|_{t=0} = & \frac{\partial^k}{\partial t^{k}} \left( \frac{1}{2\pi} \int_{-\infty}^{+\infty} e^{-i\beta x} e^{-t |\beta |}d\beta \right) \Bigg|_{t=0}\\
= & \frac{1}{2\pi} \int_{-\infty}^{+\infty} e^{-i\beta x} (-1)^k |\beta |^{k} d\beta\\
= &  \frac{(-1)^k k!}{\pi |x|^{k+1}} \cos\left( \frac{\pi (k+1)}{2} \right).
\end{align*}
\end{proof}

\begin{os}
\normalfont
We notice that for $n=1$ the problem above becomes
\begin{equation*}
\frac{\partial^{2} u}{\partial t^{2}} = - \frac{\partial^{2} u}{\partial x^{2}}, \quad x \in \mathbb{R},\; t>0
\end{equation*}
subject to the initial conditions
\begin{equation*}
\begin{cases}
u(x, 0)=\delta(x)\\
\frac{\partial u}{\partial t}(x, t)\Bigg|_{t=0^+} =  \frac{ -1}{\pi |x|^{2}}  \cos \pi
\end{cases}
\end{equation*}
which is in accord with
$$ \frac{\partial}{\partial t} \frac{t}{\pi(x^2 + t^2)} \Big|_{t=0^+} = \frac{1}{\pi x^2}. $$
The connection between wave equations and the composition of two independent Cauchy processes $C^1(|C^2(t)|)$, $t>0$ has been investigated in \citet{DO} and more general results involving the Cauchy process have been presented in \citet{NANE08}.
\end{os}

\begin{os}
\normalfont
We finally notice that the equation
\begin{equation}
\frac{\partial^6 u}{\partial t^6} + \frac{\partial^6 u}{\partial x^6} = 0
\end{equation}
can be decoupled as
\begin{equation}
\left( \frac{\partial^3}{\partial t^3} + i \frac{\partial^3}{\partial x^3} \right) \left( \frac{\partial^3}{\partial t^3} - i \frac{\partial^3}{\partial x^3} \right) u  = 0. \label{eq6ord}
\end{equation}
Form the Corollary \ref{coroAiry}, the solution to \eqref{eq6ord} can be therefore written as
\begin{align}
u(x,t) = & \frac{1}{2\pi} \left[ \frac{\frac{\sqrt{3}}{2} t e^{i\frac{\pi}{6}}}{\left( x + \frac{t e^{i \frac{\pi}{6}}}{2} \right)^2 + \frac{3}{4} t^2 e^{i \frac{\pi}{3}}} + \frac{\frac{\sqrt{3}}{2} t e^{-i\frac{\pi}{6}}}{\left( x + \frac{t e^{-i \frac{\pi}{6}}}{2} \right)^2 + \frac{3}{4} t^2 e^{-i \frac{\pi}{3}}} \right] \notag \\
= & \frac{\sqrt{3}}{2^2\pi} t \Bigg[ \frac{e^{i\frac{\pi}{6}} \left( x^2 + \frac{t^2}{4} e^{-i\frac{\pi}{3}} + xt e^{-i \frac{\pi}{6}} + \frac{3}{4}t^2 e^{-i\frac{\pi}{3}} \right) }{ \left( x^2 + \frac{t^2}{4}e^{-i\frac{\pi}{3}} + xt e^{-i \frac{\pi}{6}} + \frac{3}{4}t^2 e^{-i\frac{\pi}{3}} \right)  \left( x^2 + \frac{t^2}{4} e^{i\frac{\pi}{3}} + xt e^{i \frac{\pi}{6}} + \frac{3}{4}t^2 e^{i\frac{\pi}{3}} \right)} \nonumber \\
& + \frac{e^{-i\frac{\pi}{6}} \left( x^2 + \frac{t^2}{4} e^{i\frac{\pi}{3}} + xt e^{i \frac{\pi}{6}} + \frac{3}{4}t^2 e^{i\frac{\pi}{3}} \right)}{ \left( x^2 + \frac{t^2}{4} e^{-i\frac{\pi}{3}} + xt e^{-i \frac{\pi}{6}} + \frac{3}{4}t^2 e^{-i\frac{\pi}{3}} \right)  \left( x^2 + \frac{t^2}{4} e^{i\frac{\pi}{3}} + xt e^{i \frac{\pi}{6}} + \frac{3}{4}t^2 e^{i\frac{\pi}{3}} \right)} \Bigg] \nonumber \\
= & \frac{\sqrt{3}}{2\pi} t \frac{(x^2 + t^2)\cos \frac{\pi}{6} + xt}{\left( x^2 + t^2 e^{-i\frac{\pi}{3}} + xt e^{-i \frac{\pi}{6}} \right)  \left( x^2 + t^2 e^{i\frac{\pi}{3}} + xt e^{i \frac{\pi}{6}} \right)} \nonumber \\
= & \frac{\sqrt{3}}{2\pi} t \frac{(x^2 + t^2)\cos \frac{\pi}{6} + xt}{\left(x^2 + t^2 + xt \cos \frac{\pi}{6}\right)^2  - 3 x^2 t^2 \sin^2\frac{\pi}{6}}. \label{2cau}
\end{align}
Equation \eqref{eq6ord} is satisfied by the Cauchy density and therefore by the probability law \eqref{2cau} which however is no longer a Cauchy distribution and possesses asymmetric structure.
\end{os}

\end{document}